\documentclass[]{amsart}

\usepackage{enumitem}

\usepackage{geometry}
\usepackage[utf8]{inputenc}
\usepackage[T1]{fontenc}
\usepackage{amsmath}
\usepackage{amsthm}
\usepackage{amsfonts}
\usepackage{mathtools}
\usepackage{mathrsfs} %\mathscr
\usepackage{courier} %\texttt
\usepackage{upgreek} %\upalpha, \upbeta, ...
\usepackage{dsfont}
\usepackage{hyperref} %\url
\newcommand{\N}{\mathbb{N}}
\newcommand{\Z}{\mathbb{Z}}

\newcommand{\R}{\mathbb{R}}

\newcommand{\RNum}[1]{\uppercase\expandafter{\romannumeral #1\relax}}

\DeclarePairedDelimiterX\lrangle[2]{\langle}{\rangle}{#1, #2}

\newtheorem{theorem}{Theorem}

\numberwithin{theorem}{section}
\newtheorem{lemma}[theorem]{Lemma}
\newtheorem{proposition}[theorem]{Proposition}
\newtheorem{conjecture}[theorem]{Conjecture}
\newtheorem{corollary}[theorem]{Corollary}

\theoremstyle{definition}
\newtheorem{definition}[theorem]{Definition}

\usepackage{tikz}

\begin{document}
\title{Cops and robbers on $2K_2$-free graphs}

\author{J\'er\'emie Turcotte}
\address{D\'{e}partment de math\'{e}matiques et de statistique, Universit\'{e} de Montr\'{e}al, Montr\'eal, Canada}
\email{mail@jeremieturcotte.com}
\urladdr{www.jeremieturcotte.com}


\maketitle
\noindent

\begin{abstract}
	We prove that the cop number of any $2K_2$-free graph is at most 2, proving a conjecture of Sivaraman and Testa. We also show that the upper bound of $3$ on the cop number of $2K_1+K_2$-free (co-diamond--free) graphs is best possible.
\end{abstract}

%%%%%%%%%%%%%%%%%%%%%%%%%%%%%%%%%%%%%%%%%%%%%%%%%%%%%%%
%%%%%%%%%%%%%%%%%%%%%%%%%%%%%%%%%%%%%%%%%%%%%%%%%%%%%%%
%%%%%%%%%%%%%%%%%%%%%%%%%%%%%%%%%%%%%%%%%%%%%%%%%%%%%%%
%%%%%%%%%%%%%%%%%%%%%%%%%%%%%%%%%%%%%%%%%%%%%%%%%%%%%%%
%%%%%%%%%%%%%%%%%%%%%%%%%%%%%%%%%%%%%%%%%%%%%%%%%%%%%%%
%%%%%%%%%%%%%%%%%%%%%%%%%%%%%%%%%%%%%%%%%%%%%%%%%%%%%%%

\section{Introduction}
	Cops and robbers \cite{nowakowski_vertex--vertex_1983, quilliot_problemes_1978, aigner_game_1984} is a turn-based game opposing a group of cops to a robber on some connected graph $G$. The cops' objective is to capture the robber, whereas the latter attempts to escape indefinitely. The possible positions during the game are the vertices of $G$. On the first turn of the game, starting with the cops, each player picks the vertex where it will start the game from, and then alternate moving. When a cop or the robber is on some vertex $u$, its possible moves are staying on $u$ or moving to a vertex adjacent to $u$ (moving along an edge). The \emph{cop number} $c(G)$ \cite{aigner_game_1984} is the number of cops which is both sufficient and necessary for their victory. We say that $G$ is $k$-cop-win if $c(G)=k$ and that $G$ is $k$-cop-lose if $c(G)>k$.
	
	We define the graphs $P_t$, $C_t$, $K_t$, and $K_{t,r}$ as, respectively, the path on $t$ vertices, the cycle on $t$ vertices, the complete graph on $t$ vertices, and the complete bipartite graph with parts (colour classes) of size $t$ and $r$. If $G_1$ and $G_2$ are graphs, then $G_1+G_2$ is the disjoint union of $G_1$ and $G_2$. For any graph $G$, we define $mG$ as the graph composed of $m$ disjoint copies of $G$, that is $\underbrace{G+\dots+G}_{m \text{ times}}$, and we define $\overline G$ as the complement of $G$. Finally, we say $\N=\{1,2,\dots\}$.
	
	It is frequent in graph theory to consider excluding, or forbidding, some substructures in graphs, most notably induced subgraphs, subgraphs or minors. We will say that say a graph $G$ is $H$-free, $H$-subgraph-free or $H$-minor free if $G$ does not contain, respectively, any induced subgraph, subgraph, or minor which is isomorphic to $H$. One may similarly define graphs which exclude multiple graphs as subgraphs or minors.
	
	There has been a fair amount of research relating the cop number with forbidden subgraphs or minors. In particular, the interest has mostly been on finding constant upper bounds on the cop number (in other words, which do not vary with the order of the graph we are playing on). The first major general result of this type is the following.
	
	\begin{theorem}\cite{andreae_pursuit_1986}
		If $H$ is a graph, then there exists $M_H\in \N$ such that for any $H$-minor-free connected graph $G$ we have $c(G)\leq M_H$.
	\end{theorem}
	
	We assume that $M_H$ is as small as possible; we note that this value is denoted by $\alpha(H)$ in \cite{andreae_pursuit_1986}. With the existence of such a bound proved, one might also be interested in optimizing this value $M_H$ for specific choices of $H$. For instance, it was shown in \cite{andreae_pursuit_1986} that $M_H<|E(H)|$ if $H$ has no isolated vertices and one of its components has at least 2 edges. It is also proved in \cite{andreae_pursuit_1986} that $M_{K_5}=3$ and that $M_{K_{3,3}}=3$,  improving on the result from \cite{aigner_game_1984} that planar graphs have cop number at most 3; Wagner's theorem \cite{wagner_uber_1937} states that the class of planar graphs and the class of $\{K_5,K_{3,3}\}$-minor-free graphs coincide.
	
	Results of this type for $H$-subgraph-free and $H$-free graphs were found in \cite{joret_cops_2010}.
	
	\begin{theorem}\cite{joret_cops_2010}
		If $H$ is a graph and $S_H\in\N\cup \{\infty\}$ is smallest possible such that for any $H$-subgraph-free connected graph $G$ we have $c(G)\leq S_H$, then $S_H<\infty$ if and only if every connected component of $H$ is a tree with at most 3 vertices of degree at most 1.
	\end{theorem}
	
	\begin{theorem}\cite{joret_cops_2010}\label{thmindjoret}
		If $H$ is a graph and $I_H\in \N\cup \{\infty\}$ is smallest possible such that for any $H$-free connected graph $G$ we have $c(G)\leq I_H$, then $I_H<\infty$ if and only if every connected component of $H$ is a path.
	\end{theorem}
	
	Some families with multiple excluded induced subgraphs are discussed, for instance, in \cite{joret_cops_2010,liu_cop_2019,masjoody_cops_2020,sivaraman_cop_2019-1}. In this paper, we will consider specifically the problem of excluding one graph from being an induced subgraph, as in Theorem \ref{thmindjoret}. We want to find to bound (and if possible find the exact value of) $I_H$ from Theorem \ref{thmindjoret}. The simplest case is that of a single forbidden path as induced subgraph, for which the following bound has been proved.

	\begin{theorem}\cite{joret_cops_2010}\label{thmjoretpt}
		If $G$ is a connected $P_t$-free graph $(t\geq 3)$, then $c(G)\leq t-2$.
	\end{theorem}
	
	In other words, we know that $I_{P_t}\leq t-2$. It has been conjectured that this bound can be improved by using one fewer cops when $t\geq 5$.
	
	\begin{conjecture}\cite{sivaraman_application_2019}\label{ptconj}
		If $G$ is a connected $P_t$-free graph $(t \geq 5)$, then $c(G)\leq t-3$.
	\end{conjecture}
	
	An argument proving this statement for the class of $P_5$-free graphs might generalize to the whole conjecture. It has been suggested by Seamone and Hosseini in private communication that one possible approach towards this conjecture is first proving it for $2K_2$-free ($2P_2$-free) graphs, which is a proper subclass of $P_5$-free graphs. This conjecture first appears in \cite{sivaraman_cop_2019}.
 	
 	\begin{conjecture}\cite{sivaraman_cop_2019}\label{2k2conj}
		If $G$ is a connected $2K_2$-free graph, then $c(G)\leq 2$.
	\end{conjecture}
 	
 	The main objective of this paper is to prove this conjecture, in other words to prove that $I_{2K_2}=2$. Some partial results are obtained in \cite{liu_cop_2019,sivaraman_cop_2019}. In particular, hypothetical $2K_2$-free 2-cop-lose graphs have diameter 2 and contain induced cycles of length 3, 4 and 5, as well as an induced subgraph isomorphic to $\overline P_5$ (also called the house graph). However, we will not be using these results in our proof.
 	
 	In Section \ref{codiamondfree}, we consider the class of $2K_1+K_2$-free graphs, also known as $2P_1+P_2$-free graphs, or co-diamond--free graphs (the diamond graph being the graph $K_4$ with one edge removed). It is easy to see that these graphs have cop number at most 3. We will show that this bound is best possible by presenting an infinite family of $2K_1+K_2$-free graphs.

%%%%%%%%%%%%%%%%%%%%%%%%%%%%%%%%%%%%%%%%%%%%%%%%%%%%%%%
%%%%%%%%%%%%%%%%%%%%%%%%%%%%%%%%%%%%%%%%%%%%%%%%%%%%%%%
%%%%%%%%%%%%%%%%%%%%%%%%%%%%%%%%%%%%%%%%%%%%%%%%%%%%%%%
%%%%%%%%%%%%%%%%%%%%%%%%%%%%%%%%%%%%%%%%%%%%%%%%%%%%%%%
%%%%%%%%%%%%%%%%%%%%%%%%%%%%%%%%%%%%%%%%%%%%%%%%%%%%%%%
%%%%%%%%%%%%%%%%%%%%%%%%%%%%%%%%%%%%%%%%%%%%%%%%%%%%%%%

\section{Traps}\label{trapssection}	
	We begin with some basic notation. Let $G$ be a graph and $x\in V(G)$. We denote by $N(x)$ the neighbourhood of $x$ and by $N[x]=N(x)\cup\{x\}$ the closed neighbourhood of $u$. If $S\subseteq V(G)$, then $G-S$ denotes the subgraph of $G$ induced by $V(G)\setminus S$; if $S=\{x\}$, we write $G-x$ for $G-S$. We write that graphs $G_1,G_2$ are isomorphic by $G_1\simeq G_2$.
	
	We can now introduce an important concept which will be central in our proof.
	
\begin{definition}
	Let $G$ be a graph. A vertex $u\in V(G)$ is a \emph{trap} if there exists $x_1,x_2\in V(G)$ (not necessarily distinct) such that $x_1,x_2\neq u$ and $N[u]\subseteq N[x_1]\cup N[x_2]$. We say $u$ is trapped by $x_1,x_2$, or that $x_1,x_2$ trap $u$.
\end{definition}

In other words, a trap gives a winning position for the cops: if the robber is on a trap $u$ and the cops are on the vertices trapping $u$, then the robber cannot escape and will lose at the next turn. A trap is a generalization of the classical definition of a corner (also called an irreducible vertex) in the game with one cop, see \cite{nowakowski_vertex--vertex_1983}. We note that this concept coïncides with the concept of a $2$-trap in \cite{wagner_cops_2015} from which the terminology is inspired. The idea is also implicit in \cite{clarke_characterizations_2012}, in which one would write $u\preceq_1(x_1,x_2)$ to say that $u$ is trapped by $x_1,x_2$. We now define different types of traps.

\begin{definition}
	Let $G$ be a graph, $u\in V(G)$ be a trap and $x_1,x_2\in V(G)$ be a choice of vertices trapping $u$. We say
	\begin{enumerate}
		\item $u$ is a \emph{type-I trap} if exactly one of $x_1,x_2$ is adjacent to $u$ (in particular they are distinct);
		\item $u$ is a \emph{type-II trap} if both $x_1$ and $x_2$ are adjacent to $u$ ($x_1,x_2$ are not necessarily distinct); and
		\item $u$ is a \emph{connected trap} if $x_1,x_2$ are adjacent vertices (in particular they are distinct).
	\end{enumerate}
	
	To lighten the proofs, we will say that $u$ is \emph{c-trapped} by $x_1$ and $x_2$ in the case of a connected trap.
\end{definition}

Note that a trap can simultaneously be any combination of type-I, type-II, connected and not connected, as a vertex may be trapped in multiple ways. On the other hand, every trap must be at least one of type-I or type-II.

%%%%%%%%%%%%%%%%%%%%%%%%%%%%%%%%%%%%%%%%%%%%%%%%%%%%%%%
%%%%%%%%%%%%%%%%%%%%%%%%%%%%%%%%%%%%%%%%%%%%%%%%%%%%%%%
%%%%%%%%%%%%%%%%%%%%%%%%%%%%%%%%%%%%%%%%%%%%%%%%%%%%%%%
%%%%%%%%%%%%%%%%%%%%%%%%%%%%%%%%%%%%%%%%%%%%%%%%%%%%%%%
%%%%%%%%%%%%%%%%%%%%%%%%%%%%%%%%%%%%%%%%%%%%%%%%%%%%%%%
%%%%%%%%%%%%%%%%%%%%%%%%%%%%%%%%%%%%%%%%%%%%%%%%%%%%%%%	
	
\section{Finding connected traps}
	The structural properties of $2K_2$-free graphs have been studied in various papers, for example in \cite{chung_maximum_1990}. In this section, we prove the existence of connected traps in such graphs. We start with some well-known remarks about $2K_2$-free graphs, for which we omit the obvious proofs.
	
	\begin{lemma}\label{basicproperties}
		If $G$ is a $2K_2$-free graph, then
		\begin{enumerate}[label=(\alph*)]
			\item only one connected component of $G$ can contain edges;
			\item the diameter of any connected $2K_2$-free graph is at most 3; and
			\item any induced subgraph of $G$ is $2K_2$-free.
		\end{enumerate}
	\end{lemma}
	
	The following reformulation of the $2K_2$-free property will be used later to simplify some arguments.
	\begin{lemma}\label{cantmove}
	If $G$ is a $2K_2$-free graph,  $u\in V(G)$, $vw\in E\left(G-N[u]\right)$, then every neighbour of $u$ is adjacent to $v$ or $w$ (or both).
	\end{lemma}
	\begin{proof}
			Suppose to the contrary that there exists a neighbour $x$ of $u$, but not of $v,w$. Then, the edges $ux,vw$ form a $2K_2$.
		\end{proof}
		
	This lemma also yields a direct proof that 3 cops can catch the robber on connected $2K_2$-free graphs, as noted in \cite{sivaraman_cop_2019}. Choose an edge and place a cop on each end of this edge. By the lemma, the robber (who must choose a starting vertex not adjacent to the cops) cannot move, and a third cop can go catch the robber.
	
	\begin{lemma}\label{trap5cycle}
		If $G$ is a connected $2K_2$-free graph, $u\in V(G)$ and $G-N[u]\simeq C_5$, then $G$ contains a connected trap.
	\end{lemma}
	
	\begin{proof}
		Denote $a_1,\dots,a_5$ the vertices of $G-N[u]$, such that $a_i a_{i+1}\in E(G)$ for $1\leq i \leq 5$, working modulo 5.
	
		It is easily seen that any vertex $x\in N(u)$ must be adjacent to at least 3 vertices of the 5-cycle $G-N[u]$ by applying Lemma \ref{cantmove} for each edge $a_ia_{i+1}$.
		
		If $x$ is adjacent to 3 or more consecutive vertices of $G-N[u]$, say $a_{i-1},a_i,a_{i+1}$ for some $1\leq i\leq 5$, then $a_i$ is c-trapped by $u$ and $x$: all vertices in $G$ are dominated by $u$ or $x$, except possibly for $a_{i+2},a_{i+3}$, to which $a_i$ is not adjacent.
		
		Thus, we may now consider that every vertex of $N(u)$ is adjacent to exactly 3 vertices of the 5-cycle $G-N[u]$, only two of which are adjacent: if $x\in N(u)$, then $N(x)\setminus N[u]=\{a_i,a_{i+2},a_{i+3}\}$ for some $1\leq i\leq 5$.
		
		If distinct vertices $x_1,x_2\in N(u)$ have the same neighbours in $G-N[u]$, then $x_1$ is c-trapped by $x_2$ and $u$. Hence, we may now consider that every vertex of $N(u)$ has a distinct neighbourhood in $G-N[u]$. We will denote the \emph{possible} vertices of $N(u)$ as follows:  $N(u)\subseteq\{b_1,\dots,b_5\}$, such that $b_i$ is adjacent to exactly $a_i$, $a_{i+2}$ and $a_{i+3}$ in $G-N[u]$. If $b_i,b_{i+1}\in N(u)$, then $b_ib_{i+1}$ is an edge, as otherwise $b_ia_{i+2},b_{i+1}a_{(i+1)+3}$ would form an induced $2K_2$. This does not exclude that there may be other edges between the $b_i$'s.
		
		Choose a vertex $b_i\in N(u)$ (at least one must exist in order for $G$ to be connected). Then, $u$ is c-trapped by $b_i$ and $a_i$. Indeed, $b_i$ is adjacent to $b_{i+1}$ and $b_{i-1}$ (if they are in the graph), $a_i$ is adjacent to $b_{i+2}$ and $b_{i+3}$ (if they are in the graph), and $a_i$ and $b_i$ are adjacent. This concludes the proof.
	\end{proof}
	
	We are now ready to prove the desired result.
	\begin{proposition}\label{trapexists}
		If $G$ is a connected $2K_2$-free graph, then either
		\begin{enumerate}
			\item $G\simeq K_1$;
			\item $G\simeq K_2$;
			\item $G\simeq C_5$; or
			\item $G$ contains a connected trap.
		\end{enumerate}
	\end{proposition}
	
	\begin{proof}
		We proceed by induction. If $|V(G)|\in\{1,2\}$, this is trivially true. Suppose that $|V(G)|\geq 3$ and the statement is true by induction for connected $2K_2$-free graphs $G'$ such that $|V(G')|<|V(G)|$. Let $u$ be any vertex of $G$. Recall that $G-N[u]$ is $2K_2$-free, by Lemma \ref{basicproperties}$(c)$.
		
		If $G-N[u]$ is empty, then $u$ dominates $G$. As $|V(G)|\geq 3$, the vertex $u$ has at least two distinct neighbours $x_1,x_2$. Then, $x_1$ is c-trapped by $u$ and $x_2$.
		
		If $G-N[u]$ contains a connected component which is a single vertex $y$, then $y$ is c-trapped by $u$ and any neighbour of $y$ (which is necessarily in $N(u)$). Otherwise, $G-N[u]$ contains no isolated vertex and by Lemma \ref{basicproperties}(a), $G-N[u]$ is connected. Also, $G-N[u]$ contains more than one vertex.
		
		If $G-N[u]$ is an edge $v_1v_2$: If $v_1$ and $v_2$ have a common neighbour $t$ in $N(u)$, then $v_1$ is c-trapped by $t$ and $u$. Otherwise, $v_1$ and $v_2$ have no common neighbour. Denote by $A$ the neighbours of $v_1$ in $N(u)$ and by $B$ the neighbours of $v_2$ in $N(u)$. By Lemma \ref{cantmove}, $N(u)=A\cup B$. At least one of $A,B$ must be non-empty in order for $G$ to be connected. Without loss of generality, $|A|\geq |B|$. If $|A|=1$ and $|B|=0$, then $G$ is path of length 4, which contains a connected trap. If $|A|=|B|=1$, then we either have that $G\simeq C_5$ (if the vertex in $A$ and the vertex in $B$ are not adjacent) or $G$ contains a connected trap (if the vertex in $A$ and the vertex in $B$ are adjacent, $u$ is a connected trap). Now consider that $|A|>1$ and let $a_1,a_2\in A$ be distinct vertices. As $a_1$ and $a_2$ are both adjacent to $v_1$ but not $v_2$, we have that $a_1$ is c-trapped by $a_2$ and $u$.
		
		The remaining case is that $G-N[u]$ contains at least 3 vertices (and is connected). By the inductive hypothesis, $G-N[u]$ is either a $C_5$ or contains a connected trap. If $G-N[u]\simeq C_5$, then Lemma \ref{trap5cycle} yields that $G$ contains a connected trap. Otherwise, denote by $v$ the vertex of $G-N[u]$ which is a connected trap, and $w_1,w_2$ the vertices trapping $v$ in $G-N[u]$. We know that $w_1,w_2$ together dominate $v$ and all neighbours of $v$ in $G-N[u]$. As $w_1w_2\in E(G)$, they also dominate all vertices in $N(u)$ by Lemma \ref{cantmove}. Hence, $v$ is also a connected trap in $G$.
	\end{proof}
	
%%%%%%%%%%%%%%%%%%%%%%%%%%%%%%%%%%%%%%%%%%%%%%%%%%%%%%%
%%%%%%%%%%%%%%%%%%%%%%%%%%%%%%%%%%%%%%%%%%%%%%%%%%%%%%%
%%%%%%%%%%%%%%%%%%%%%%%%%%%%%%%%%%%%%%%%%%%%%%%%%%%%%%%
%%%%%%%%%%%%%%%%%%%%%%%%%%%%%%%%%%%%%%%%%%%%%%%%%%%%%%%
%%%%%%%%%%%%%%%%%%%%%%%%%%%%%%%%%%%%%%%%%%%%%%%%%%%%%%%
%%%%%%%%%%%%%%%%%%%%%%%%%%%%%%%%%%%%%%%%%%%%%%%%%%%%%%%
	
\section{A winning strategy}
	In this section, we improve the upper bound on the cop number of $2K_2$-free graphs by using the traps found in the previous section.
	
	In general, the fact that a graph contains traps does not necessarily imply that the cops can bring the game to that position. For example, it is shown in \cite{pisantechakool_conjecture_2017} that all planar graphs of order at most 19 contain a trap, but it is still open as to whether 2 cops can win on all planar graphs of order at most 19. Another example is that it is shown in \cite{wagner_cops_2015} that all diameter 2 graphs of order $n$ contain a set of vertices of size at most $\sqrt{n}$ which dominates the neighbourhood of some other vertex (called a $\sqrt{n}$-trap), but it is unknown whether the cop number of these graphs is upper bounded by $\sqrt{n}$ (it is proved to be bounded by $\sqrt{2n}$). In our case however, we will show that containing a trap will give us meaningful information. This is somewhat similar to the equivalence between cop-win and dismantlable graphs, see \cite{nowakowski_vertex--vertex_1983}.
	
	For the remainder of this section, we will denote by $\widehat G$ a minimal (meaning smallest relative to the number of vertices) connected $2K_2$-free 2-cop-lose graph. Our objective is to find a contradiction in order to show that such a graph cannot exist. We first need the following lemma.

\begin{lemma}\label{noisolatedverticesrestricted}
	For any $u\in V(\widehat G)$, the induced subgraph $\widehat G-u$ is connected, and the induced subgraph $\widehat G-N[u]$ is non-empty, connected and contains no isolated vertex.
\end{lemma}

\begin{proof}
	Recall that any induced subgraph of $\widehat G$ is $2K_2$-free, by Lemma \ref{basicproperties}$(c)$. If $\widehat G-u$ is disconnected, then by Lemma \ref{basicproperties}$(a)$, there is a vertex $x$ isolated in $\widehat G-u$. This implies that in $\widehat G$, the only neighbour of $x$ is $u$. It is easily seen that removing a vertex of degree 1 does not change the cop number of a graph nor make it disconnected. This contradicts the minimality of $\widehat G$, as $\widehat G-x$ would be a connected $2K_2$-free 2-cop-lose graph on fewer vertices.

	It is clear $\widehat G-N[u]$ is non-empty, otherwise a single cop on $u$ would catch the robber in 1 turn, contradicting that $\widehat G-N[u]$ is 2-cop-lose. Suppose there exists a vertex $v$ which is isolated in $\widehat G-N[u]$. Then, $v$ is such that all of its neighbours in $\widehat G$ are in $N(u)$. As $\widehat G-v$ is a connected $2K_2$-free graph on fewer vertices than $\widehat G$, there exists a winning strategy for 2 cops on $\widehat G-v$. 
	
	We can then define a winning strategy for 2 cops on $\widehat G$ using the strategy on $\widehat G-v$. We say the robber's shadow is on $u$ whenever the robber is actually on $v$, and for all other positions the robber's shadow is on the same vertex as the robber. Now, as $N(v)\subseteq N(u)$, any move the robber makes yields a valid move for the robber's shadow on $\widehat G-v$. The cops apply the strategy on $\widehat G-v$ to catch the robber's shadow. At the end of this strategy, if the robber is not caught, then necessarily the robber is on $v$ and a cop is on $u$. This cop stays on $u$, and the robber on $v$ cannot move. The other cop may then go capture the robber, contradicting that $\widehat G$ is 2-cop-lose. This is a well known argument, see \cite[Theorem 1]{nowakowski_vertex--vertex_1983} and \cite[Theorems 3.1 and 3.2]{berarducci_cop_1993} for more general versions.
	
	Since $\widehat G-N[u]$ contains no isolated vertex, it is connected by Lemma \ref{basicproperties}$(a)$.
\end{proof}

Even if 2 cops cannot win on $\hat G$, we now show that the cops have great power as to making the robber move to a given position. To some extent we will be able to assume that we can place both the cops and the robber.

\begin{lemma}\label{placingplayers}
	If $u\in V(\widehat G)$ and $vw\in E(\widehat G-N[u])$, then there exists a strategy, playing with 2 cops, ensuring that the cops are on $v,w$ and the robber is on $u$ and cannot move.
\end{lemma}
\begin{proof}
	We first wish to force the robber to move to $u$. By Lemma \ref{noisolatedverticesrestricted}, $\widehat G-u$ is connected, and by Lemma \ref{basicproperties}$(c)$ it is $2K_2$-free. Hence, by the minimality of $\widehat G$, we know that $\widehat G-u$ has cop number at most 2. As long as the robber is not on $u$, the cops copy the strategy for $\widehat G-u$ on $\widehat G$. If the robber never moves to $u$, the robber will eventually be caught, hence the robber has no choice but to eventually move to $u$. Denote $z_1$ and $z_2$ the positions of the cops at that point, we know that $z_1,z_2\notin N[u]$, as otherwise the cops could capture the robber one turn later, a contradiction as $\widehat G$ is 2-cop-lose.
	
	We now wish to bring the two cops to the ends of any edge in $\widehat G-N[u]$, while keeping the robber on $u$. If $z_1=z_2$, one of the cops moves to a neighbour of $z_1$ in $\widehat G-N[u]$, which must exist as $\widehat G-N[u]$ is not a unique vertex by Lemma \ref{noisolatedverticesrestricted}. If $z_1z_2\in E(\widehat G)$, then they are already in the desired position. If $z_1$ and $z_2$ have a common neighbour $z$ in $\widehat G-N[u]$, we move the cop on $z_2$ to $z$. If not, by Lemma \ref{noisolatedverticesrestricted}, $\widehat G-N[u]$ is connected and, by Lemma \ref{basicproperties}(b), $z_1$ and $z_2$ are at distance 3 in $\widehat G-N[u]$: there exists $z_1',z_2'$ such that $z_1z_1'z_2'z_2$ is a path contained in $\widehat G-N[u]$. We move the cop on $z_1$ to $z_1'$ and the cop on $z_2$ to $z_2'$. Now that the cops are on adjacent vertices, both not in $N[u]$, then by Lemma \ref{cantmove}, the robber cannot move.
	
	We now wish to bring the cops to the ends of the edge $vw$, while keeping the robber on $u$. We will do so by never leaving $\widehat G-N[u]$ and always keeping the cops on adjacent vertices, which guarantees that the robber will never be able to move. Suppose the cops are now on the edge $ab$. Let $P$ be a path completely contained in $\widehat G-N[u]$ starting with the edge $ab$ and ending with the edge $vw$, which exists as $\widehat G-N[u]$ is connected. The cops move along $P$ one behind the other. This concludes the proof.
\end{proof}

In Section \ref{trapssection}, we defined type-I and type-II traps. Using the strategy we developed in the last lemma, we will be able to exclude these from $\widehat G$.

\begin{lemma}\label{notypeI}
	$\widehat G$ does not contain a type-I trap.
\end{lemma}
\begin{proof}
	Suppose to the contrary that there exists a type-I trap $u$. We will use this trap to construct a winning strategy for $2$ cops on $\widehat G$.
	
	Let $x,v$ be the vertices trapping $u$, with $x$ adjacent to $u$ and $v$ in $\widehat G-N[u]$. Let $w$ be any neighbour of $v$ in $\widehat G-N[u]$, which exists as $\widehat G-N[u]$ contains no isolated vertex by Lemma \ref{noisolatedverticesrestricted}. Using Lemma \ref{placingplayers}, place the cops on $v$ and $w$, and the robber on $u$.
	
	If $wx$ is an edge, then move the cop on $w$ to $x$ and keep the other cop on $v$. If $wx$ is not an edge, then $xv$ is an edge by Lemma \ref{cantmove}. Move the cop on $v$ to $x$ and the cop on $w$ to $v$.
	
	In both cases, the robber is now on $u$ with the cops on $x,v$: the robber is caught at the next move. This is a contradiction as $\widehat G$ is 2-cop-lose.
\end{proof}

Before considering the case of type-II traps, we need the following proposition from \cite{chung_maximum_1990}. We prove it here in order for this paper to be self-contained.
\begin{proposition}\cite{chung_maximum_1990}\label{chungbipartite}
	If $G$ is a connected bipartite $2K_2$-free graph, then each colour class of $G$ contains a vertex adjacent to all vertices of the other colour class of $G$.
\end{proposition}
\begin{proof}
	Denote $A,B$ the colour classes of $G$. Choose $m\in A$ of highest degree. Suppose there exists $b\in B$ such that $mb$ is not an edge. As $G$ is connected, there exists $a\in A$ such that $ab$ in an edge. Now, for every neighbour $x\in N(m)$ (necessarily, $x\in B$), we compare edges $ab$ and $mx$: the $2K_2$-free property yields that $ax$ is an edge. Thus, $|N(a)|>|N(m)|$, since $N(m)\subseteq N(a)$ and $b\in N(a)\setminus N(m)$, which contradicts that $m$ has highest degree in $A$.
\end{proof}

\begin{lemma}\label{typeIIimpliestypeI}
	If $\widehat G$ contains a type-II trap, then $\widehat G$ contains a type-I trap.
\end{lemma}

\begin{proof}
	Let $x_1$,$x_2$ be the vertices trapping a vertex $u$ such that $x_1$ and $x_2$ are both adjacent to $u$. We can suppose $x_1$ and $x_2$ are distinct, as if $N[u]\subseteq N[x_1]$, then simply pick $x_2$ to be any other neighbour of $u$ (which must exist as otherwise $\widehat G-x_1$ is disconnected, contradicting Lemma \ref{noisolatedverticesrestricted}).
	
	Suppose $w$ is a neighbour of $x_1$ in $\widehat G-N[u]$, we wish to prove $w$ is adjacent to $x_2$. Suppose $w$ is not adjacent to $x_2$, then denote by $v$ any neighbour of $w$ in $\widehat G-N[u]$, which exists as $\widehat G-N[u]$ contains no isolated vertex by Lemma \ref{noisolatedverticesrestricted}. Then, $v$ must adjacent to $x_2$ by Lemma \ref{cantmove}. Playing with 2 cops, place the cops on $w$ and $v$ and the robber on $u$ using Lemma \ref{placingplayers}. Then, move the cop on $w$ to $x_1$ and the cop on $v$ to $x_2$. The robber will be caught one turn later, which is a contradiction as $\widehat G$ is 2-cop-lose. Thus, $w$ must be adjacent to $x_2$.
	
	By applying this reasoning for every neighbour of $x_1$ and of $x_2$ in $\widehat G-N[u]$, we find that every vertex of $\widehat G-N[u]$ is either adjacent to both $x_1$ and $x_2$, or to neither. We can thus partition $V(\widehat G)\setminus N[u]$ into the sets $A=\{v\in V(\widehat G)\setminus N[u] : zx_1, zx_2\in E(\widehat G)\}$ and $B=\{v\in V(\widehat G)\setminus N[u] : zx_1, zx_2\notin E(\widehat G)\}$.
	
	If there is an edge between 2 vertices in $B$, comparing this edge with $ux_1$ yields an induced $2K_2$, and thus $B$ is a stable set. If there is an edge between two vertices in $A$, then, playing with 2 cops, place the cops on the ends of this edge and the robber on $u$, using Lemma \ref{placingplayers}, and then move the cops to $x_1$ and $x_2$, yielding a contradiction as $\widehat G$ is 2-cop-lose. Thus, $\widehat G-N[u]$ is a (connected, by Lemma \ref{noisolatedverticesrestricted}) bipartite graph. Note that $B$ is non-empty as $A$ is a stable set and $\widehat G-N[u]$ contains no isolated vertex.
	
	By Proposition \ref{chungbipartite}, there exists a vertex $b$ in $B$ adjacent to every vertex of $A$. Every neighbour of $x_1$ in $N[u]$ is (by definition) either $u$ or adjacent to $u$, and every neighbour of $x_1$ in $\widehat G-N[u]$ is adjacent to $b$. Furthermore, $x_1b\notin E(\widehat G)$. Thus, $x_1$ is a type-I trap, trapped by $u$ and $b$.
\end{proof}

We are now ready to prove Conjecture \ref{ptconj}.

\begin{theorem}\label{maintheorem}
	If $G$ is a connected $2K_2$-free graph, then $c(G)\leq 2$.
\end{theorem}
	
	\begin{proof}
		Let $\widehat G$ be a minimal counter-example. Lemmas \ref{notypeI} and \ref{typeIIimpliestypeI} imply that $\widehat G$ does not contain any trap, hence does not contain any connected trap. Thus, by Proposition \ref{trapexists}, $\widehat G$ is isomorphic to either $K_1$, $K_2$ or $C_5$, all of which have cop number at most 2. Hence, no minimal counter-example to the statement exists, which proves the statement.
	\end{proof}
	
	We note that this result is best possible (in other words $I_{2K_2}=2$), as there are an infinite number of $2K_2$-free graphs with cop number 2. It is easily seen that complete multipartite graphs are $2K_2$-free, and that such graphs have cop number 2 if each colour class (each independent set) has size at least 2.
	
	The more general question of the cop number of $mK_2$-free graphs ($m\geq2$) is raised in \cite{sivaraman_cop_2019}. One easily notices that $2m-1$ cops can win, as noted in \cite{sivaraman_cop_2019}: place $2$ cops on the ends of an edge $uv$ and apply induction to the $(m-1)K_2$-free graph $G-N[u]\cup N[v]$, with the base case being $2K_2$-free graphs. Having improved by 1 the bound on the cop number of $2K_2$-free graphs, we can also improve by 1 the bound on the cop number of $mK_2$-free graphs.
	
	\begin{corollary}\label{rk2free}
		If $G$ is a connected $mK_2$-free graph $(m\geq 2)$, then $c(G)\leq 2m-2$.
	\end{corollary}
	
	Applying the same type of argument with base case Theorem \ref{thmjoretpt} yields the following general result as formulated in \cite{liu_cop_2019}.
	
	\begin{theorem}\cite{liu_cop_2019}
		If $G$ is a $P_{i_1}+\dots+P_{i_k}$-free graph $(k,i_1,\dots,i_k\in \N)$ and there is at least one index $i_j\geq 3$ $(1\leq j\leq k)$, then $c(G)\leq i_1+\dots+i_k-2$.
	\end{theorem}
	
	Using Theorem \ref{maintheorem}, we can then loosen the restriction on the indices as follows.
	
	\begin{corollary}\label{generalsum}
		If $G$ is a $P_{i_1}+\dots+P_{i_k}$-free graph $(k,i_1,\dots,i_k\in \N)$ and there is either
		\begin{enumerate}
			\item at least one index $i_j\geq 3$ $(1\leq j\leq k)$, or
			\item there are two indices $i_j,i_{j'}=2$ $(1\leq j<j'\leq k)$,
		\end{enumerate}
		then $c(G)\leq i_1+\dots+i_k-2$.
	\end{corollary}
	
%%%%%%%%%%%%%%%%%%%%%%%%%%%%%%%%%%%%%%%%%%%%%%%%%%%%%%%
%%%%%%%%%%%%%%%%%%%%%%%%%%%%%%%%%%%%%%%%%%%%%%%%%%%%%%%
%%%%%%%%%%%%%%%%%%%%%%%%%%%%%%%%%%%%%%%%%%%%%%%%%%%%%%%
%%%%%%%%%%%%%%%%%%%%%%%%%%%%%%%%%%%%%%%%%%%%%%%%%%%%%%%
%%%%%%%%%%%%%%%%%%%%%%%%%%%%%%%%%%%%%%%%%%%%%%%%%%%%%%%
%%%%%%%%%%%%%%%%%%%%%%%%%%%%%%%%%%%%%%%%%%%%%%%%%%%%%%%

\section{Co-diamond--free graphs}	\label{codiamondfree}
	The statement of Corollary \ref{generalsum} naturally leads us to asking whether the restriction on the indices can be further loosened by only requiring one of the indices to be at least 2. We are thus interested in the case of $mP_1+P_2$-free (or $mK_1+K_2$-free) graphs.
	
	The statement does not generalize for the case $m=1$. It is well-known \cite{de_ridder_et_al_cop_3--free_nodate} and easy to see that the $P_1+P_2$-free graphs (or $\overline P_3$-free graphs) are exactly the complete multipartite graphs. As noted earlier, the cop number of such graphs is 2 when each colour class has size at least 2.
	
	In this section, we will consider the case $m=2$: $2K_1+K_2$-free (or co-diamond--free graphs).
	\begin{proposition}\label{uppbound2k1k2}
		If $G$ is a connected $2K_1+K_2$-free graph, then $c(G)\leq 3$.
	\end{proposition}
	\begin{proof}
		The proof is completely analogous to the one for $2K_2$-free graphs. If $G$ is complete, then $c(G)=1$. Otherwise, there exists non-adjacent vertices $v,w$; place one cop on each of these vertices. The robber must choose a vertex $u$ adjacent to neither $v$ nor $w$. Since $G$ is $2K_1+K_2$-free, all neighbours or $u$ are adjacent to $v$ or $w$, hence the robber cannot move. A third cop can then go capture the robber.
	\end{proof}
	
	It would be natural to attempt proving that the cop number of these graphs is at most $2$ using a proof similar to the one for $2K_2$-free graphs, but replacing connected traps by \emph{disconnected trap}, that is a trap for which the trapping vertices are not adjacent (and distinct). Although Lemmas \ref{notypeI} and \ref{typeIIimpliestypeI}	 appear to have analogues for $2K_1+K_2$-free graphs (we note this is using some ideas from \cite[Lemma 4.4]{baird_minimum_2014}), no inductive strategy to find traps similar to Lemma \ref{trapexists} can work. Indeed, if $u$ is a vertex of a $2K_1+K_2$-free graph $G$, then $G-N[u]$ is a $K_1+K_2$-free graph, so necessarily complete multipartite graph. If each colour class has size at least 3, it does not contain any disconnected trap. Graphs for which this holds for every vertex $u$ (if such graphs exist) are good candidates to have cop number 3. A computer search using the \texttt{geng} program from the \texttt{nauty Traces} package \cite{mckay_practical_2014} (using the \texttt{PRUNE} feature to generate $2K_1+K_2$-free graphs), \texttt{Mathematica} \cite{wolfram_research_inc_mathematica_nodate} and two implementations of a cop number algorithm \cite{afanassiev_cop-number_2017,turcotte_code_2020}, we were able to find a $2K_1+K_2$-free graph requiring 3 cops, shown in Figure \ref{fig:3copexample}.
	
	\begin{figure}[h]
	\begin{tikzpicture}[scale=1.4, dot/.style = {circle, fill, minimum size=#1,
					inner sep=0pt, outer sep=0pt}, dot/.default = 5pt
			% size of the circle diameter 
		]

		%   Outside
		\node [dot] (1) at (0,0) {};
		\node [dot] (2) at (0,1) {};
		\node [dot] (3) at (0,2) {};
		\node [dot] (4) at (0,3) {};
		\node [dot] (5) at (1,3) {};
		\node [dot] (6) at (1,2) {};
		\node [dot] (7) at (1,1) {};
		\node [dot] (8) at (1,0) {};
		\node [dot] (9) at (2,0) {};
		\node [dot] (10) at (2,1) {};
		\node [dot] (11) at (2,2) {};
		\node [dot] (12) at (2,3) {};
		\node [dot] (13) at (3,3) {};
		\node [dot] (14) at (3,2) {};
		\node [dot] (15) at (3,1) {};
		\node [dot] (16) at (3,0) {};
		
		\draw (1) to (5);
		\draw (1) to (6);
		\draw (1) to (7);
		\draw (1) to (10);
		\draw (1) to (11);
		\draw (1) to (12);
		\draw (1) to (13);
 		\draw (1) to (14);
 		\draw (1) to (15);
		\draw (2) to (5);
		\draw (2) to (6);
		\draw (2) to (8);
		\draw (2) to (9);
		\draw (2) to (11);
		\draw (2) to (12);
		\draw (2) to (13);
		\draw (2) to (14);
		\draw (2) to (16);
		\draw (3) to (5);
		\draw (3) to (7);
		\draw (3) to (8);
		\draw (3) to (9);
		\draw (3) to (10);
		\draw (3) to (12);
		\draw (3) to (13);
		\draw (3) to (15);
		\draw (3) to (16);
		\draw (4) to (6);
		\draw (4) to (7);
		\draw (4) to (8);
		\draw (4) to (9);
		\draw (4) to (10);
		\draw (4) to (11);
		\draw (4) to (14);
		\draw (4) to (15);
		\draw (4) to (16);
		\draw (5) to (9);
		\draw (5) to (10);
		\draw (5) to (11);
		\draw (5) to (14);
		\draw (5) to (15);
		\draw (5) to (16);
		\draw (6) to (9);
		\draw (6) to (10);
		\draw (6) to (12);
		\draw (6) to (13);
		\draw (6) to (15);
		\draw (6) to (16);
		\draw (7) to (9);
		\draw (7) to (11);
		\draw (7) to (12);
		\draw (7) to (13);
		\draw (7) to (14);
		\draw (7) to (16);
		\draw (8) to (10);
		\draw (8) to (11);
		\draw (8) to (12);
		\draw (8) to (13);
		\draw (8) to (14);
		\draw (8) to (15);
		\draw (9) to (13);
		\draw (9) to (14);
		\draw (9) to (15);
		\draw (10) to (13);
		\draw (10) to (14);
		\draw (10) to (16);
		\draw (11) to (13);
		\draw (11) to (15);
		\draw (11) to (16);
		\draw (12) to (14);
		\draw (12) to (15);
		\draw (12) to (16);
	\end{tikzpicture}
	\caption{$K_4\times K_4$: All pairs of vertices are adjacent except when aligned vertically or horizontally. Some edges overlap in the drawing.}
	\label{fig:3copexample}
\end{figure}
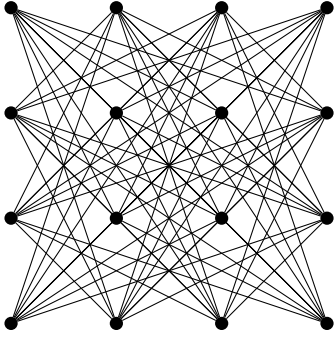

We can in fact generalize this example to an infinite and well-known class of graphs. We will need some definitions. We write $[n]=\{1,\dots,n\}$. Let $G_1,G_2$ be graphs. We define their \emph{categorical product} (or \emph{tensor product}) as the graph with vertex set $V(G_1)\times V(G_2)$ and for which $(u_1,u_2),(v_1,v_2)\in V(G_1)\times V(G_2)$ are adjacent if $u_1v_1\in E(G_1)$ and $u_2v_2\in E(G_2)$. The cop number of products of graphs (categorical product as well as cartesian and strong products) is in particular studied in \cite{neufeld_game_1998}.

We are interested in graphs of the form $K_a\times K_b$, which are known as the complement of Rook's graphs (see \cite{weisstein_rook_nodate,wikipedia_rooks_nodate,wikipedia_tensor_nodate}, for instance). These graphs are easily visualized in $\R^2$ as the points (vertices) $[a]\times[b]$ where all vertices are pairwise adjacent except when they lie on a line parallel to one of the axes (in other words, they differ in both coordinates). It can also be seen as the Cayley graph of the group $\left(\Z/a\Z\times \Z/b\Z,+\right)$ with the generating set $$S=\{(z_1,z_2)\in \Z/a\Z\times \Z/b\Z:z_1,z_2\neq 0\}.$$

Let us note that for every vertex $u$ of $G\simeq K_a\times K_b$, we have that $G-N[u]$ is a bipartite graph with colour classes of sizes $a-1$ and $b-1$.
	
\begin{theorem}
	If $G\simeq K_a\times K_b$ for some $a,b\in \N$, then
	\begin{enumerate}
		\item $G$ is $2K_1+K_2$-free; and
		\item \cite[Theorem 3.2]{neufeld_game_1998} if $a,b\geq 4$, then $c(G)=3$.
	\end{enumerate}
\end{theorem}
\begin{proof}\item 
	\begin{enumerate}
		\item We work using the embedding of the graph in $\R^2$ discussed above. Suppose the contrary, say that $x_1,x_2,x_3,x_4$ are distinct vertices all pairwise non-adjacent except for the edge $x_3x_4$. As $x_1$ and $x_2$ are not adjacent, there exists a line $L$ containing $x_1$ and $x_2$ which is parallel to one of the axes. Similarly, let $L'$ be the line containing $x_1$ and $x_3$. If $L\neq L'$, then $L$ and $L'$ are perpendicular, one being horizontal and one vertical. Then, the line between $x_2$ and $x_3$ is diagonal (formally, it has non-zero and non-infinite slope), which is a contradiction since these vertices are not adjacent. Hence, $L$ contains $x_3$. Analogously, $L$ contains $x_4$. This would imply that $x_3$ and $x_4$ are non-adjacent, which is a contradiction.
		\item We prove this well-known statement here in order for this paper to be self-contained. The fact that $c(G)\leq 3$ follows from Proposition \ref{uppbound2k1k2} and part (1) of the statement. We show that if there are only 2 cops the robber always has an escape from the cops (and by a same argument, that there is a safe starting position for the robber), hence that $c(G)\geq 3$.
		
		Without loss of generality, suppose $G=G_1\times G_2$, where $G_1\simeq K_a, G_2\simeq K_b$. Suppose the cops are on vertices $v=(v_1,v_2),w=(w_1,w_2)$ and the robber is on $u=(u_1,u_2)$. If neither $v$ nor $w$ are adjacent to $u$, the robber can stay still. Otherwise, without loss of generality $v$ is adjacent to $u$.
		
		Suppose $v,w$ differ in at most 1 coordinate, without loss of generality in the first coordinate (in other words $v_2=w_2$). Let $x\in V(G_1)\setminus \{u_1,v_1,w_1\}$, which exists since $|V(G_1)|=n_1\geq 4$. The robber can move to $r=(x,v_2)$ as $r$ and $u$ differ in both coordinates: for the first coordinate by our choice of $x$, and in the second coordinate since $u$ is adjacent to either $v$ or $w$. Since $r$ differs only in one coordinate from $v$ and $w$, the robber is now safe.
		
		Suppose $v,w$ differ in both coordinates. We would like to move the robber to either $(v_1,w_2)$ or $(w_1,v_2)$, which are the (only) vertices not adjacent to either $v$ or $w$. If $u$ is not adjacent to $(v_1,w_2)$ nor $(w_1,v_2)$, then $u_1=v_1$ or $u_2=w_2$, and $u_1=w_1$ or $u_2=v_2$. In other words, the only vertices adjacent to neither $(v_1,w_2)$ or $(w_1,v_2)$ (apart from themselves) are exactly $v,w$, and so the robber can move safely to one of these vertices.
		
		In all cases, the robber can win against 2 cops, which completes the proof.
	\end{enumerate}
\end{proof}

%%%%%%%%%%%%%%%%%%%%%%%%%%%%%%%%%%%%%%%%%%%%%%%%%%%%%%%
%%%%%%%%%%%%%%%%%%%%%%%%%%%%%%%%%%%%%%%%%%%%%%%%%%%%%%%
%%%%%%%%%%%%%%%%%%%%%%%%%%%%%%%%%%%%%%%%%%%%%%%%%%%%%%%
%%%%%%%%%%%%%%%%%%%%%%%%%%%%%%%%%%%%%%%%%%%%%%%%%%%%%%%
%%%%%%%%%%%%%%%%%%%%%%%%%%%%%%%%%%%%%%%%%%%%%%%%%%%%%%%
%%%%%%%%%%%%%%%%%%%%%%%%%%%%%%%%%%%%%%%%%%%%%%%%%%%%%%%
	
\section{Further directions}	
	Containing a winning position (in our case, a trap), is a necessary condition for the cops' victory, but it is in general not sufficient. It was however crucial in the proof of our main result to first show that this necessary condition held. We believe that this approach could be useful to bound the cop number of other graph classes, at least when they are very structured. The following is a list of open problems in this direction, more specifically for graphs with forbidden induced subgraphs.
	
	\begin{enumerate}
		\item It would be interesting to see if the approach used to prove Theorem \ref{maintheorem} can be used to improve the bound on the cop number of $P_5$-free graphs, and even possibly to prove Conjecture \ref{ptconj}, which was our initial motivation.
		\item Can we further improve the bound on the cop number of $mK_2$-free graphs ($m\geq 2$) or is Corollary \ref{rk2free} best possible?
		\item Can the cop number of $mK_1+K_2$-free graphs ($m\geq 3$) be as large as $m+1$ (as in the cases $m=1,2$)?
		\item Denote by $\gamma(G)$ the domination number of $G$ (the size of the smallest dominating set in $G$) and by $\alpha(G)$ the independence number of $G$ (the size of the largest stable set in $G$). It is easily seen and well known that $c(G)\leq \gamma(G)\leq \alpha(G)$. What can we say about the cop number of $mK_1$-free graphs ($m\geq 4$), in other words graphs such that $\alpha(G)\leq m-1$? Can we improve the upper bound to $c(G)\leq m-2$ in order to further generalize Corollary \ref{generalsum}? If not, what can we say about graphs for which $c(G)=\gamma(G)=\alpha(G)$?
	\end{enumerate}	
	
	We note that improving these upper bounds also has a practical application. Computing the cop number of a graph is quite costly in computation time; currently the best algorithm \cite{clarke_characterizations_2012} to determine whether an $n$-vertex graph has cop number at most $k$ has complexity $O\left(n^{2k+2}\right)$. It may then be useful to first verify if there is some easy to compute parameter which directly yields a bound on the cop number. Verifying whether the graph is $H$-free, for small graphs $H$ such as $2K_2$, can be one of these tests. There is some interest in the community for improving practical computation of the cop number (even without improving the worst case complexity of the algorithm); the author has discussed the topic with various people in the past and this application was also suggested by one of the anonymous reviewers.

\section*{Acknowledgments}
	I thank Seyyed Aliasghar Hosseini and Ben Seamone for introducing me to this problem and for their helpful comments and suggestions, including the suggestion of adding Corollary \ref{rk2free} and the idea of considering minimality. I am very grateful to Samuel Yvon for our discussions, which have inspired me greatly. I thank Ge\v na Hahn and Simon St-Amant for their comments on the paper. I express my gratitude to Aliz\'ee Gagnon for noticing an oversight in a previous version of this paper. Finally, I thank the anonymous reviewers for their suggestions for this paper.

	The author is supported by the Natural Sciences and Engineering Research Council of Canada (NSERC) and the Fonds de Recherche du Qu\'ebec -- Nature et technologies (FRQNT). L'auteur est supporté par le Conseil de recherches en sciences naturelles et en génie du Canada (CRSNG) et le Fonds de Recherche du Québec -- Nature et technologies (FRQNT).
	
\bibliography{refs.bib}
\bibliographystyle{abbrvurl}

\end{document}